\pdfoutput=1
\documentclass[twoside,leqno,twocolumn]{article}
\usepackage{ltexpprt}

\def\final{1}
\usepackage{pdfsync}
\usepackage{color}
\usepackage{xspace,wrapfig,subfigure}
\usepackage{url}
\usepackage{graphicx}
\usepackage{amsmath, amsthm, amssymb}

\usepackage{amsfonts}

\usepackage{paralist}

\usepackage{enumerate}

\ifnum\final=0
\newcommand{\mynote}[1]{\marginpar{\tiny\sf #1}}
\else
\newcommand{\mynote}[1]{}
\fi

\usepackage{theomac}
\newtheoremstyle{mytheoremstyle} 
{\topsep}                    
{\topsep}                    
{\itshape}                   
{}                           
{\scshape}                   
{.}                          
{.5em}                       
{}  
\theoremstyle{mytheoremstyle}
\newtheoremWithMacro{theorem}{Theorem}[section]

\newtheorem{lemma}{Lemma}[section]

\newtheorem{proposition}{Proposition}[section]

\newtheorem{Conjecture}{Conjecture}[section]

\newcommand{\alg}[2]{\begin{center}\fbox{\begin{minipage}{0.99\columnwidth}{\begin{center}\underline{\textsc{#1}}\end{center}{#2}}\end{minipage}}\end{center}}

\def \E {\mathbb{E}}

\newcommand{\e}{\mathrm{e}}
\newcommand{\da}{\mathrm{d}a}

\newcommand{\ds}{\mathrm{d}s}
\newcommand{\dt}{\mathrm{d}t}

\newcommand{\ddelta}{\mathrm{d}\delta}

\newcommand{\figref}[1]{Figure \ref{fig:#1}}
\newcommand{\lemref}[1]{Lemma \ref{lem:#1}}
\newcommand{\propref}[1]{Proposition \ref{prop:#1}}
\newcommand{\theoref}[1]{Theorem \ref{theo:#1}}
\newcommand{\secref}[1]{Section \ref{sec:#1}}

\renewcommand{\paragraph}{\subsection}
\begin{document}

\title{\Large The Rigidity Transition in Random Graphs}
\author{Shiva Prasad Kasiviswanathan\thanks{IBM T.\ J.\ Watson Research Center, Yorktown Heights. Work done while the author was as a postdoc at Los Alamos National Laboratory, kasivisw@gmail.com} \and Cristopher Moore\thanks{Santa Fe Institute and University of New Mexico, moore@santafe.edu} \and Louis Theran\thanks{Temple University, theran@temple.edu.  Supported by CDI-I grant DMR 0835586 to I. Rivin and M. M. J. Treacy.}}

\date{}

\maketitle

\begin{abstract} \small\baselineskip=9pt
As we add rigid bars between points in the plane, at what point is there a giant (linear-sized) rigid component,
which can be rotated and translated, but which has no internal flexibility?  If the points are generic, this depends only on the combinatorics
of the graph formed by the bars.  We show that if this graph is an Erd\H{o}s-R\'enyi random graph $G(n,c/n)$, then there exists a sharp
threshold for a giant rigid component to emerge. For $c < c_2$, w.h.p. all rigid components span one, two, or three vertices,
and when $c > c_2$, w.h.p. there is a giant rigid component.  The constant $c_2 \approx 3.588$ is the threshold for $2$-orientability,
discovered independently by Fernholz and Ramachandran and Cain, Sanders, and Wormald in SODA'07.  We also give quantitative bounds on
the size of the giant rigid component when it emerges, proving that it spans a $(1-o(1))$-fraction of the vertices in the $(3+2)$-core. Informally, the $(3+2)$-core is maximal induced subgraph obtained by starting from the $3$-core and then inductively adding vertices with $2$ neighbors in the graph obtained so far.

\end{abstract}

\section{Introduction} \label{sec:intro}

Imagine we start with a set of $n$ points allowed to move freely in the Euclidean plane and add fixed-length bars between pairs of the points, one at a time.  Each bar fixes the distance between its endpoints, but otherwise does not constrain the motion of the points.

Informally, a maximal subset of the points which can rotate and translate, but otherwise has no internal flexibility is called a \emph{rigid component}. As bars are added, the set of rigid components may change, and this change can be very large: the addition of a single bar may cause $\Omega(n)$ many rigid components spanning $O(1)$ points to merge into a single component spanning $\Omega(n)$ points.

We are interested in the following question: \emph{If we add bars uniformly at random at what point does a giant (linear-sized) rigid component emerge and what is its size?}  Our answers are: (1) there is a phase transition from all components having at most three points to a unique giant rigid component when about $1.794n$ random bars are added; (2) when the linear-sized rigid component emerges, it contains at least nearly all of the $3$-core of the graph induced by these bars.

One of the major motivations for studying this problem comes from physics, where these planar bar-joint frameworks (formally described below)  are used to understand the physical properties of systems such as bipolymers and glass networks (see, e.g., the book by Thorpe~\emph{et al.}\ \cite{ThJaChRa99}).

A sequence of papers~\cite{JaTh95,jacobs:hendrickson:PebbleGame:1997a,thorpe:rigidity:glasses:2002,chubynsky2002rigid,ThJaChRa99} studied the emergence of large rigid components in glass networks generated by various stochastic processes, with the edge probabilities and underlying topologies used to model the temperature and chemical composition of the system.  An important observation that comes from these results is that {\em very large} rigid substructures emerge {\em very rapidly}. Of particular relevance to this paper are the results of~\cite{rivoire2006exactly,moukarzel2003rigidity,ThJaChRa99}. Through numerical simulations they show that that there is a sudden emergence of a giant rigid component in the $3$-core of a $G(n,p)$ random graph. The simulations of Rivoire and Barr{\'e} (see Figure 1 in~\cite{rivoire2006exactly}) also show that this phase transition occurs when there are about $1.794n$ edges in the $3$-core. Our results confirm these observations {\em theoretically}.

\paragraph{The Planar Bar-Joint Rigidity Problem.}  The formal setting for the problem described above is the well-studied \emph{planar bar-joint framework} model from rigidity theory (see, e.g., \cite{graver:servatius:rigidityBook:1993} for an overview and complete definitions).  A \emph{bar-joint framework} is a structure made of \emph{fixed-length bars} connected by \emph{universal joints} with full rotational freedom at their endpoints.  The allowed continuous motions preserve the lengths and connectivity of the bars.  A framework is \emph{rigid} if the only allowed motions are rotations or translations (i.e., Euclidean motions); it is \emph{minimally rigid} if it is rigid but ceases to be so if any bar is removed. If the framework is not rigid, it decomposes uniquely into \emph{rigid components}, which are the inclusion-wise maximal rigid sub-frameworks.  Figure \ref{fig:component-examples} shows examples of rigid components.

\begin{figure}[htbp]
\centering

\subfigure[]{\includegraphics[width=0.35\textwidth]{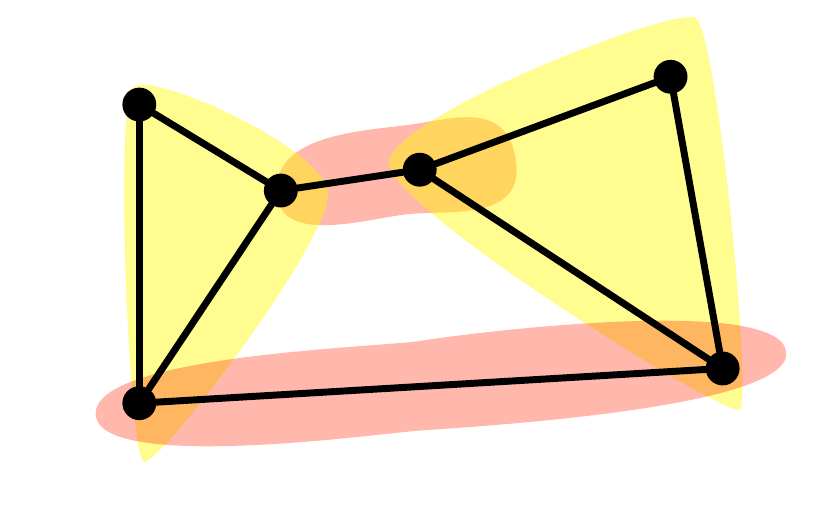}}\vspace{0.1 in}
\subfigure[]{\includegraphics[width=0.35\textwidth]{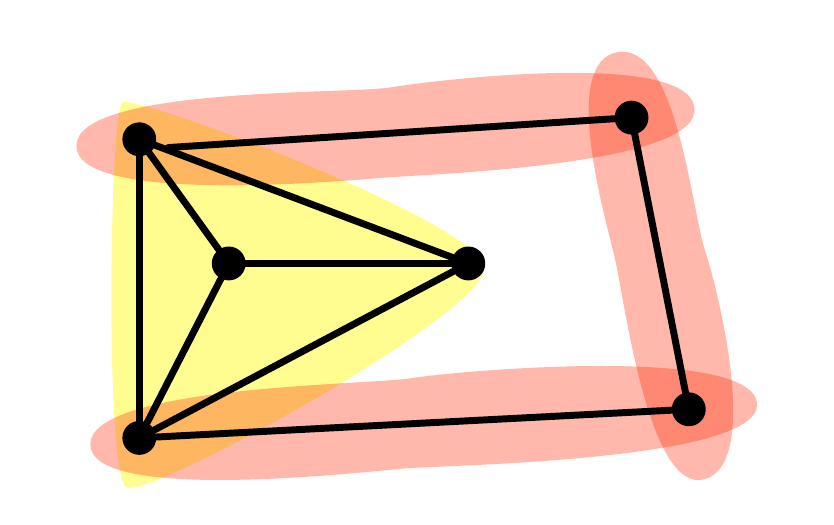}}
\caption{Examples of rigid components: (a) all the rigid components have size $2$ or $3$; (b) the indicated induced $K_4$ is rigid, but has one more edge than is required for minimal rigidity.}
\label{fig:component-examples}
\vspace{-0.1 in}
\end{figure}

The combinatorial model for a bar-joint framework is a simple graph $G=(V,E)$ with $n$ vertices representing the joints,
and $m$ edges representing the bars.
A remarkable theorem of Maxwell-Laman \cite{L70,M64} says that rigidity for a \emph{generic}\footnote{Genericity
is a subtle concept that is different than the standard assumption of \emph{general position} that appears in the
computational geometry literature.  See \cite{sliders} for a more detailed discussion.} framework (and {\em almost all}
frameworks are generic) is determined by the underlying graph alone. The graph-theoretic condition characterizing minimal
rigidity is a \emph{hereditary sparsity count}; for minimal rigidity the graph $G$ should have $m=2n-3$ edges, and every
subgraph induced by $n'$ vertices in $G$ should have at most $2n'-3$ edges.
Therefore, by the Maxwell-Laman Theorem generic rigidity in the plane becomes a \emph{combinatorial} concept,
and from now on we will consider it as such. Full definitions are given in \secref{rigid-prelim}.

\paragraph{Contributions.} With this background, we can restate our main question as follows:
\emph{What is the behavior of rigid components in a generic framework with its combinatorics given by an
Erd\H{o}s-R\'enyi random graph $G(n,p)$}?  Our main result is the following:
\begin{theorem}[\maintheorem][Main Theorem]\label{theo:main}
For any constant $c>0$ the following holds,
\begin{itemize}
\item If $c < c_2$, then, w.h.p., all rigid components in $G(n,c/n)$ span at most three vertices and
\item If $c > c_2$, then, w.h.p., there is a unique giant rigid component in
$G(n,c/n)$ spanning a $(1-o(1))$ fraction of the $(3+2)$-core.
\end{itemize}
\end{theorem}
The $(3+2)$-core of $G(n,c/n)$ is the maximal induced subgraph obtained by starting from the $3$-core
and then inductively adding vertices with $2$ neighbors in the graph obtained so far (see \secref{random-prelim} for the full definition). The constant $c_2\approx 3.588$ is the threshold for $2$-orientability discovered independently by Fernholz and Ramachandran~\cite{FR07} and Cain,
Sanders, and Wormald~\cite{CSW07}.  A graph $G$ is $2$-orientable if all its edges can be oriented so that each vertex has
out-degree\footnote{We use out-degree for consistency with the pebble game \cite{streinu:lee:pebbleGames:2008}.
In \cite{FR07,CSW07},
$2$-orientability is defined in terms of \emph{in-degree} exactly two orientations.} at most two.
There is a natural connection between the notions of $2$-orientability and minimal rigidity:
$2$-orientable graphs can be characterized using a counting condition that closely resembles the counting condition of minimal
rigidity (see \secref{rigid-prelim}).
This connection explains intuitively why the threshold for the emergence of giant rigid component should be at least $c_2$.
For example, if a giant rigid component emerges with $c < c_2$, then addition of another $o(n)$ random edges would create, with high probability,
a ``locally dense'' induced subgraph with more than twice the number of edges than vertices.
This prevents the graph from being $2$-orientable, contradicting the $2$-orientability threshold theorems of~\cite{FR07,CSW07}.

We prove the bound on the size of the giant rigid component by showing this following result.

\begin{theorem}[\almosttheorem] \label{theo:almost}
Let $c > c_2$ be a constant.  Then, w.h.p., there is a subgraph of the $(3+2)$-core such that the edges of
this subgraph can be oriented to give all but $O(\log^3 n\sqrt{n})$ of the vertices in the $(3+2)$-core an out-degree of two.
\end{theorem}

The results of~\cite{FR07,CSW07} show that for $c>c_2$, with high probability $G(n,c/n)$ is not $2$-orientable,
they don't give quantitative bounds on the size of the set of vertices in $G(n,c/n)$ that can be guaranteed an out-degree~$2$.
Theorem~\ref{theo:almost}, achieves this goal. Our proof for Theorem~\ref{theo:almost} is constructive and uses an extension of
the $2$-orientability algorithm of Fernholz and Ramachandran \cite{FR07}.  Our analysis is quite different from~\cite{FR07} and is
based on proving subcriticality of the various branching processes generated by our algorithm. We use differential equations to
model the branching process, and show subcriticality by analyzing these differential equations.

\paragraph{Other Related Work.} Jackson \emph{et al.}\ \cite{servatius:random:2008} studied the space of random
$4$-regular graphs and showed that with high probability they are {\em globally rigid} (a stronger notion of rigidity \cite{C05}).
In the $G(n,p)$ model
they prove that when $p=n^{-1}(\log n+2\log\log n+\omega(1))$, then with high probability $G(n,p)$ is rigid, but they have no results for
$G(n,p)$ when the expected number of edges is $O(n)$. In a recent result, Theran~\cite{LT09} showed using a simple counting argument that for
a constant $c$ w.h.p.\ all rigid components in $G(n,c/n)$ are either {\em tiny} or {\em giant}. Since we use this result as a technical tool,
it is introduced in more detail in \secref{random-prelim}.

\paragraph{Organization.} This paper is organized as follows.  We introduce the required background in
combinatorial rigidity in \secref{rigid-prelim} (rigidity experts may skip this section),
and the technical tools from random graphs we use to
prove \theoref{main} in \secref{random-prelim} (random graphs experts may skip this section).
With the background in place, we prove some graph theoretic
lemmas in \secref{graphs}. The proof that \theoref{almost} implies \theoref{main} is in
\secref{thresh}.

The remainder of the paper is devoted to the proof of \theoref{almost}.  \secref{conf} introduces the
facts about the random configuration model we need, and then we present our
$2$-orientation algorithm in \secref{algo}.  \secref{almost} proves \theoref{almost}.

\paragraph{Notations.} Throughout this paper $G$ is a graph $(V,E)$ with $|V|=n$ and $|E|=m$. All our graphs are simple unless explicitly
stated otherwise.  Subgraphs are typically denoted by $G'$ with $n'$ vertices and $m'$ edges.  Whether a subgraph is
edge-induced or vertex-induced is always made clear.  A \emph{spanning subgraph} is one that includes the entire vertex set $V$.

Erd\H{o}s-R\'enyi random graphs on $n$ vertices with edge probability $p$ are denoted $G(n,p)$.  Since we are interested in random graphs with constant
average degree, we use the parameterization $G(n,c/n)$, where $c>0$ is a fixed constant.

\paragraph{Asymptotics.} We are concerned with the asymptotic behavior of $G(n,c/n)$ as $n\to \infty$.
The constants implicit in the $O(\cdot)$, $\Omega(\cdot)$, $\Theta(\cdot)$; and the convergence implicit in $o(\cdot)$ are all
taken to be uniform.  A sequence of events
$\mathcal{E}_n=(E_n)_{n=1}^\infty$ holds \emph{with high probability} (shortly \emph{w.h.p.})
if $\Pr\left[E_n\right] = 1 - o(1)$.

\section{Rigidity preliminaries} \label{sec:rigid-prelim}

In this section, we introduce the notations of and a number of standard results on \emph{combinatorial rigidity}
that we use throughout.  All of the (standard) combinatorial lemmas presented here can be established by the methods of (and are cited to)
\cite{streinu:lee:pebbleGames:2008,maps}, but we give some proofs for completeness and to introduce non-experts to
style of combinatorial argument employed below.

\paragraph{Sparse and Spanning Graphs.} A graph $G$ with $n$ vertices and $m$ edges is $(k,\ell)$-sparse if, for all edge-induced
subgraphs on $n'$ vertices and $m'$ edges, $m'\le kn'-\ell$.  If, in addition $m=kn-\ell$, $G$ is $(k,\ell)$-tight.  If $G$
has a $(k,\ell)$-tight spanning subgraph it is $(k,\ell)$-spanning.  When $k$ and $\ell$ are non-negative integer parameters
with $\ell\in [0,2k)$ the $(k,\ell)$-sparse graphs form a matroidal family \cite[Theorem 2]{streinu:lee:pebbleGames:2008}
with rich structural properties,
some of which we review below. In the interest of brevity, we introduce only the parts of the theory required.

In particular, throughout, we are interested in only two settings of the parameters $k$ and $\ell$: $k=2$ and $\ell=3$; and $k=2$ and
$\ell=0$.  For economy, we establish some standard terminology following \cite{streinu:lee:pebbleGames:2008}.  A $(2,3)$-tight
graph is defined to be a \emph{Laman graph}; a $(2,3)$-sparse graph is
\emph{Laman-sparse}; a $(2,3)$-spanning graph is \emph{Laman-spanning}.

\paragraph{The Maxwell-Laman Theorem and Combinatorial Rigidity.}  The terminology of Laman graphs is motivated by the
following remarkable theorem of Maxwell-Laman.
\begin{proposition} (Maxwell-Laman~\cite{M64,L70}) \label{thm:ML}
A generic bar-joint framework in the plane is minimally rigid if and only if its graph is a Laman graph.
\end{proposition}
An immediate corollary is that a generic framework is rigid, but not necessarily minimally so, if and only if its
graph is Laman-spanning.  From now on, we will switch to the language of sparse graphs, since our setting is
\emph{entirely combinatorial}.

\paragraph{Rigid Blocks and Components.}
Let $G$ be a graph.  A \emph{rigid block} in $G$ is defined to be a vertex-induced Laman-spanning
subgraph.  We note that if a block is not an induced Laman graph, then there may be many different choices of edge sets certifying that it is
Laman spanning.  A \emph{rigid component} of $G$ is an inclusion-wise maximal block.\footnote{Readers familiar with \cite{streinu:lee:pebbleGames:2008}
will notice that our definition is slightly different, since we allow graphs that are not $(k,\ell)$-sparse.}
As a reminder to the reader, although we have retained the standard terminology of \emph{``rigid''} components,
these definitions are graph theoretic.

A \emph{Laman-basis} of a graph $G$ is a maximal subgraph of $G$ that is Laman-sparse.  All of these are the same size
by the matroidal property of Laman graphs \cite[Theorem 2]{streinu:lee:pebbleGames:2008}, and each rigid block in $G$
induces a Laman graph on its vertex set in any Laman basis of $G$.  Thus, we are free to pass through to a Laman
basis of $G$ or any of its rigid components without changing the rigidity behavior of $G$.

We now present some properties of rigid blocks and components that we use extensively.
\begin{lemma}[{\cite[Theorem 5]{streinu:lee:pebbleGames:2008}}]\label{lem:component-decomp}
Any graph $G$ decomposes uniquely into rigid components, with each edge in exactly one rigid component.
These components intersect pairwise on at most one vertex, and they are independent of the choice of Laman basis for $G$
\end{lemma}
\begin{proof}
Since a single edge forms a rigid block, each edge must be in a maximal rigid block, which is the definition of a components.  By
picking a Laman basis of $G$, we may assume, w.l.o.g., that $G$ is Laman-sparse.  In that case, it is easy to check that two
rigid blocks intersecting on at least two vertices form a larger rigid block.  Since components are rigid blocks, we then
conclude that components intersect on at most one vertex.  The rest of the lemma then follows from edges having two endpoints.
\end{proof}

\begin{lemma}[{\cite[Theorem 2]{streinu:lee:pebbleGames:2008}}]\label{lem:monotone}
Adding an edge to a graph $G$ never decreases the size of any rigid component in $G$; i.e., rigidity is a monotone property
of graphs.
\end{lemma}
\begin{proof}
There are two cases: either the new edge has both endpoints in a rigid component of $G$ or it does not.  In the first
case, the component was already a Laman-spanning induced subgraph and remains that way.  In the second case, \lemref{component-decomp}
implies that the new edge is in exactly one component of the new graph; this may subsume other components of $G$ or be just the new
edge.  Either way, all of the components of $G$ remain rigid blocks in the new graph.
\end{proof}

The following lemma is quite well-known.
\begin{lemma}
\label{lem:blocks-combine}
Let $G$ be a graph, and let $G_1=(V_1,E_1)$ and $G_2=(V_2,E_2)$ be rigid blocks in $G$ and suppose that either:
\begin{itemize}
\item $V_1\cap V_2=\emptyset$ and there are at least three edges with one endpoint in $V_1$ and the other in $V_2$, and
these edges are incident on at least two vertices in $V_1$ and $V_2$
\item $V_1\cap V_2\neq\emptyset$ (and so by \lemref{component-decomp} the intersection is a single vertex $v$)
and there is one edge $ij$ with $i\in V_1$, $j\in V_2$ and $i$ and $j$ distinct from $v$
\end{itemize}
Then $V_1\cup V_2$ is a rigid block in $G$.
\end{lemma}
\begin{proof}
There are two cases to check.  In either case, by \lemref{component-decomp} it is no loss of generality to assume that
$G_1$ and $G_2$ are Laman graphs on $n_1$ and $n_2$ vertices.  The stated result follows from picking bases.

If $V_1$ and $V_2$ are disjoint, we further assume that there are exactly three
edges going between them.  Call this set $E_3$. Since the $E_i$ are disjoint by \lemref{component-decomp},
we see that $V_1\cup V_2$ spans $2(n_1+n_2) - 3$ total edges.  Taking an arbitrary subset $V'\subset V_1\cup V_2$ of $n'$
vertices, we see that it spans at most $|E_1\cap E(V'\cap V_1)| + |E_2\cap E(V'\cap V_2)| + 3 = 2n'-3$ edges, proving that
$V_1\cup V_2$ spans an induced Laman graph.

The cases where $V_1$ and $V_2$ is similar, after accounting for a one-vertex overlap with inclusion-exclusion.
\end{proof}

\begin{lemma}[{\cite[Corollary 6]{streinu:lee:pebbleGames:2008}}]
\label{lem:circuits}
If $G$ is a simple graph on $n$ vertices and has $m>2n-3$ edges, then $G$ spans a rigid block that is not Laman-sparse
on at least four vertices.  This block has minimum vertex degree at least~$3$.
\end{lemma}
\begin{proof}
Since $G$ has more than $2n-3$ edges, it is not Laman-sparse.  Select an edge-wise minimal subgraph $G'$ on $n'$ vertices and $m'$
edges that is not Laman-sparse, and, additionally, make $n'$ minimum.  By minimality of $m'$, $m'=2n'-2$, and since $G$ is simple, $G'$
is not a doubled edge, and thus has at least four vertices.  Since dropping any edge from $G'$ results in a subgraph on the same
vertices with $2n'-3$ edges that is Laman-sparse, $G'$ is Laman-spanning, giving the desired rigid block.  Finally, removing a
degree one or two vertex from $G'$ would result in a smaller subgraph that is not Laman-sparse, so minimality of $n'$ implies that
$G'$ has minimum vertex degree $3$.
\end{proof}

\begin{lemma}[{\cite[Lemma 4]{streinu:lee:pebbleGames:2008}}]
\label{lem:laman-spanning-degrees}
If $G=(V,E)$ is Laman-spanning graph on $n$ vertices, then $G$ has minimum degree at least two.
\end{lemma}
\begin{proof}
Pick a Laman basis $G'$ of $G$.	If $G'$ has a degree one vertex $v$, then $V-v$ spans $2n-4>2(n-1)-3$
edges, contradicting the assumption that $G'$ was a Laman graph.  Thus no graph $G$ with a degree one
vertex can have a spanning subgraph that is a Laman graph.
\end{proof}

\begin{lemma}[{\cite[Lemma 17]{streinu:lee:pebbleGames:2008}}]
\label{lem:peel-degree-two}
If $G=(V,E)$ is a Laman-spanning graph on $n$ vertices,
removing a degree two vertex results in a smaller Laman-spanning graph.
\end{lemma}
\begin{proof}
Let $v$ be a degree two vertex in $G$.  Pick a Laman basis $G'$ of $G$.  By \lemref{laman-spanning-degrees},
both edges incident on $v$ are in $G'$.  In $G'$, $V-v$ spans $2n-3-2=2(n-1)-3$ edges, implying that $V-v$
induces a smaller Laman graph in $G'$, from which it follows that $V-v$ is Laman-spanning.
\end{proof}

\paragraph{$2$-orientatbility and $(2,0)$-sparsity.}  We now consider the structure properties of $(2,0)$-sparse graphs.  The
properties we review here can be obtained from \cite{maps}.  A graph $G$ is defined to be $2$-orientable if its
edges can be oriented such that each vertex has out-degree at most $2$.  There is a close connection between
$2$-orientability and $(2,0)$-sparsity expressed in the following lemma.

\begin{lemma}[{\cite[Lemma 6]{maps} or \cite[Theorem 8 and Lemma 10]{streinu:lee:pebbleGames:2008}}]
\label{lem:mapsparse}
A graph $G$ is $(2,0)$-tight if and only if it is a maximal $2$-orientable graph.
\end{lemma}
\begin{proof}[Proof sketch]
If $G=(V,E)$ is maximal and $2$-orientable, is has $n$ vertices and $2n$ edges.  Counting edges by their tails in an out-degree at most
two orientation, any subset
of $n'$ vertices is incident on, and therefore induces, at most $2n'$ edges.  On the other hand, the sparsity
counts and Hall's Matching Theorem implies that the bipartite graph with vertex classes indexed by $E$ and two copies of $V$ with
edges between ``edge vertices'' and the copies of their endpoints has a perfect matching.  The matching yields the desired
orientation by orienting edges into the vertex they are matched to, as there are two copies of every vertex in the bipartite graph.
\end{proof}
As a corollary, we obtain,
\begin{lemma}[{\cite[Theorem 8 and Lemma 10]{streinu:lee:pebbleGames:2008}}]
\label{lem:sparseorientable}
A graph $G$ is $2$-orientable if and only if it is $(2,0)$-sparse.
\end{lemma}
\begin{proof}
If $G$ is $2$-orientable, than any subset $V'$ of $n'$ vertices is incident on, and thus induces, at most $2n'$ edges.  On
the other hand, if $G$ is $(2,0)$-sparse, extend it to being $(2,0)$-tight and then apply \lemref{mapsparse} to get the
required orientation.
\end{proof}

\paragraph{Henneberg Moves and $2$-orientability.}
\begin{figure}[htbp]
\centering
\subfigure[]{\includegraphics[width=0.45\textwidth]{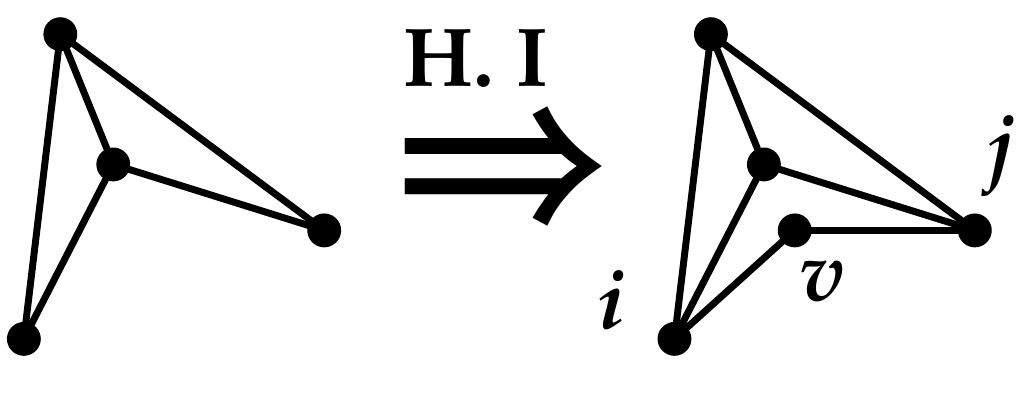}}
\subfigure[]{\includegraphics[width=0.45\textwidth]{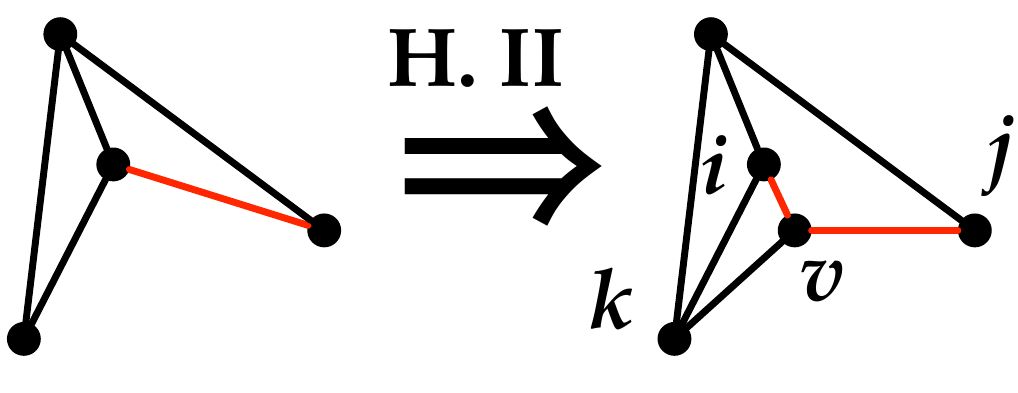}}
\caption{Examples of the Henneberg moves: (a) Henneberg I; (b) Henneberg II.  In (b), the
edge that is ``split,'' and the two new edges that replace it, are shown in red.}
\label{fig:henneberg}
\end{figure}
In our analysis of the $2$-orientation heuristic, we will make use
of so-called Henneberg moves, which give inductive characterizations of all $(k,\ell)$-sparse graphs.
Henneberg moves originate from \cite{hen} and are generalized to the entire family of $(k,\ell)$-sparse
graphs in \cite{streinu:lee:pebbleGames:2008}.  The moves are defined as follows:
\begin{description}
\item[\textbf{Henneberg I:}] Let $G$ be a graph on $n-1$ vertices.  Add a new vertex $v$ to $G$ and two
new edges to neighbors $i$ and $j$.
\item[\textbf{Henneberg II:}] Let $G$ be a graph on $n-1$ vertices, and let $ij$ be an edge in $G$.
Add a new vertex $v$ to $G$, select a neighbor $k$, remove the edge $ij$, and add edges between $v$ and $i$, $j$, and $k$.
\end{description}
\figref{henneberg} shows examples of the two moves.
Since we are concerned with $(2,0)$-sparsity, while $v$ must be new, the neighbors $i$, $j$, and $k$ may be the same as each other or $v$.  When
this happens we get self-loops or multiple copies of the same edge (i.e., a multigraph).  The
fact we need later is,
\begin{lemma}[{\cite[Lemma 10 and Lemma 17]{streinu:lee:pebbleGames:2008} or \cite{FR07}\footnote{Under the name ``excess degree reduction.''}}]
\label{lem:henneberg}
The Henneberg moves preserve $2$-orientability.
\end{lemma}
We give a proof for completeness, since we will use the proof idea later.
\begin{proof}
Assume $G$ is $2$-orientable.  For the Henneberg I move, orient the two new edges out of the new vertex $v$.  For the Henneberg II move,
suppose that $ij$ is oriented $i\to j$ in $G$.  Orient the new edges $i\to v$, $v\to j$, $v\to k$.
\end{proof}
We remark that although the development here follows along the lines
of the rigidity-inspired \cite{streinu:lee:pebbleGames:2008}, this idea was developed (to our knowledge)
independently by Fernholz and Ramachandran \cite{FR07}.

\paragraph{Almost Spanning Subgraphs.}  We introduce a final piece of notation, which is the
concept of an \emph{almost spanning graph}.  A graph $G$ on $n$ vertices
is defined to be \emph{almost $(2,0)$-spanning} if it contains a $(2,0)$-spanning subgraph on $n-o(n)$
vertices.

\paragraph{The Explosive Growth of Rigid Components.}
\begin{figure}[htbp]
\centering
\subfigure[]{\includegraphics[width=0.3\textwidth]{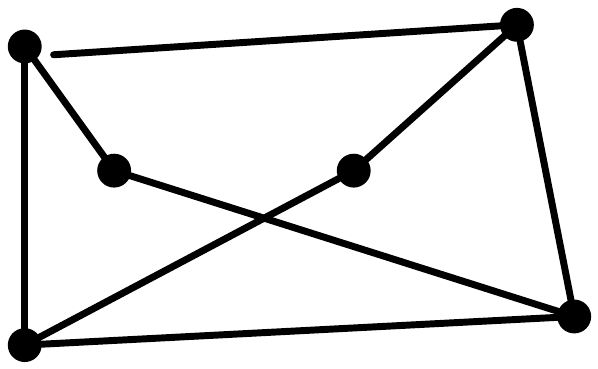}}
\subfigure[]{\includegraphics[width=0.3\textwidth]{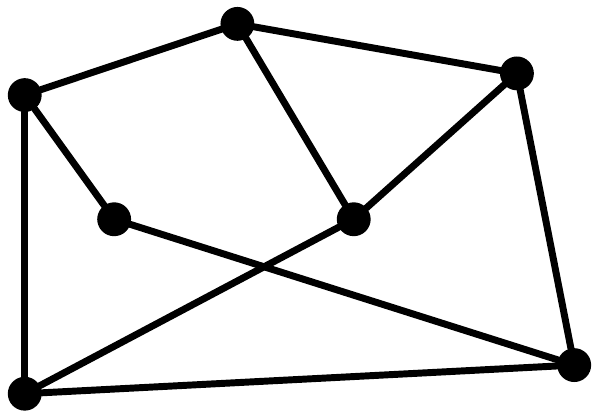}}
\subfigure[]{\includegraphics[width=0.3\textwidth]{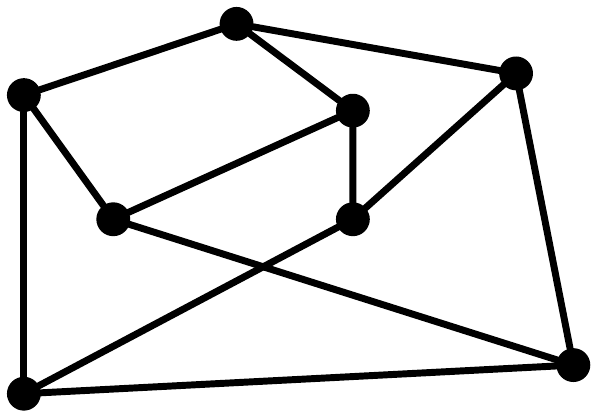}}
\caption{A family of graphs in which adding any edge rigidifies the entire graph.}
\label{fig:no-components-examples}
\end{figure}
We conclude this section with an example that shows how the behavior of rigid components can be \emph{very} different than
that of connectivity.  Unlike connectivity, the size of the largest rigid component in a graph may increase from $O(1)$ to $\Omega(n)$
after adding only one edge.  A dramatic family of examples is due to Ileana Streinu \cite{IS}.
We begin with the graph obtained by dropping an edge from $K_{3,3}$; this graph has $2 \cdot 6-4=8$ edges, and its
rigid components are simply the edges.  We then repeatedly apply the Henneberg II move, avoiding triangles whenever possible.
As we increase the number of vertices $n$, for even $n$, we obtain graphs with $2n-4$ edges and no rigid components spanning more than
two vertices (see Figure~\ref{fig:no-components-examples}(c)); but adding \emph{any} edge to these graphs results in a Laman-spanning graph.

This example can be interpreted as saying that rigidity is an inherently non-local phenomenon.

\section{Random graphs preliminaries} \label{sec:random-prelim}

Let $G(n,p)$ be a random graph on $n$ vertices where each edge appears independently of all others with probability $p$.
In this paper, we are interested in sparse random graphs, which are generated by $p=c/n$ for some constant $c$. This section
introduces the results from random graphs that we need as technical tools, along with our \theoref{almost} to prove the
main Theorem~\ref{theo:main}.

\paragraph{Size of the $3$-core.}
The\emph{ $k$-core} of a graph is defined as the maximal induced subgraph of minimum degree at least $k$.
The $k$-core thresholds for random graphs have been studied in~\cite{PSW,C04,M04,JL07}.
For $\mu >0$, let $Po(\mu)$ denote a Poisson random variable with mean $\mu$. Let us denote the Poisson probabilities by $\pi_j(\mu) = \Pr[Po(\mu)=j]$. Let $\psi_j(\mu) =  \Pr[Po(\mu) \geq j]$. Also, let $\lambda_k = \min_{\mu > 0}\mu/\psi_{k-1}(\mu).$
For $\lambda > \lambda_k$, let $\mu_k(\lambda) > 0$ denote the largest solution to
$\mu/\psi_{k-1}(\mu)=\lambda$. In~\cite{PSW}, Pittel, Spencer and Wormald discovered that for
$k \geq 3$, $\lambda=\lambda_k$ is the threshold for the appearance of a nonempty $k$-core in the
random graph $G(n,\lambda/n)$.

\begin{proposition}[Pittel, Spencer, and Wormald~\cite{PSW}] \label{prop:PSW}
Consider the random graph $G(n,\lambda/n)$, where $\lambda > 0$ is fixed. Let $k \geq 2$ be fixed. If $\lambda < \lambda_k$, then w.h.p.
the $k$-core is empty in $G(n,\lambda/n)$.
If  $\lambda > \lambda_k$, then w.h.p.
there exists a $k$-core in $G(n,\lambda/n)$ whose size is about $\psi_k(\mu_k(\lambda))n$.
\end{proposition}

We will use the $k=3$ instance of \propref{PSW}. Substituting $k=3$ in the above equation gives
$\lambda_3 \approx 3.351$, and the size of the 3-core at $\lambda=\lambda_3$ is about $0.27n$.

\paragraph{The $(3+2)$-core.}
Extending the $3$-core, we define the $3+2$-core of a graph as the maximal subgraph that can be
constructed starting from the $3$-core and inductively adding vertices with at least two neighbors in the
subgraph built so far.

From the definition, it is easy to see that the $(3+2)$-core emerges when the $3$-core does, and,
since it contains the $3$-core, it it, w.h.p., empty or giant in a $G(n,p)$.  A branching
process heuristic indicates that, after it emerges, the fraction of the vertices in the $(3+2)$-core
is the root $q$ of the equation $q = 1-\e^{-qc} (1+qc)$ (where $\e^{-qc} (1+qc)$ comes from
$\Pr[Po(qc) < 2]$).  However, we do not know how to make the argument rigorous, and so we leave it as a conjecture.
At $c=3.588$ the conjectured number of vertices in the $(3+2)$-core is about $0.749n$.

\paragraph{The $2$-orientability Threshold.} Define the constant $c_2$ to be the
supremum of $c$ such that the $3$-core of  $G(n,c/n)$ has average degree at most $4$.

Fernholz and Ramachandran and Cain, Sanders, and Wormald independently proved that $c_2$ is the threshold
for $2$-orientability of $G(n,c/n)$.

\begin{proposition}[\cite{FR07,CSW07}] \label{prop:FR}
With high probability, for any fixed constant $c>0$:
\begin{itemize}
\item If $c < c_2$, $G(n,c/n)$ is $2$-orientable.
\item If $c > c_2$, $G(n,c/n)$ is not $2$-orientable.
\end{itemize}
\end{proposition}

\paragraph{Rigid Components are Small or Giant.} An edge counting argument ruling out small
dense subgraphs in $G(n,c/n)$ shows the following fact about rigid components in random
graphs.
\begin{proposition}[Theran~\cite{LT09}] \label{prop:LT}
Let $c > 0$ be a constant. Then, w.h.p., all rigid components in a random graph $G(n,c/n)$ have size $1$, $2$, $3$,
or $\Omega(n)$.
\end{proposition}

\section{Graph-theoretic lemmas} \label{sec:graphs}
This short section gives the combinatorial lemmas we need to derive \theoref{main} from \theoref{almost}.
The first is a simple observation relating adding edges to the span of a rigid component and $2$-orientability.
\begin{lemma}\label{lem:adding-edges-to-laman-blocks}
Let $G=(V,E)$ be a graph and let $G'=(V',E')$ be a rigid component of $G$.  Then after adding any $4$ edges
to the span of $V'$, the resulting graph is not $2$-orientable.
\end{lemma}
\begin{proof}
Let $V'$ have $n'$ vertices.  Since $G'$ is a rigid component, it has a Laman basis by definition and thus spans at least $2n'-3$ edges.   After
the addition of $4$ edges to the span of $V'$, it spans at least $2n'+1$ edges, blocking $2$-orientability by \lemref{sparseorientable}.
\end{proof}

Another simple property we will need is that if $G$ is not Laman-sparse it spans a rigid component with non-empty $(3+2)$-core.
\begin{lemma}\label{lem:component-3plus2core}
Let $G$ be a simple graph that is not Laman-sparse.  Then $G$ spans a rigid component on at least four vertices
that is contained in the $(3+2)$-core.
\end{lemma}
\begin{proof}
Since $G$ is simple and fails to be Laman sparse, the hypothesis of \lemref{circuits} is met, so there is a
rigid block in $G$ with a non-empty $3$-core.  The component containing this block has minimum degree
two by \lemref{laman-spanning-degrees}, and peeling off degree two vertices will never result
in a degree one vertex by \lemref{peel-degree-two}, so it is in the $(3+2)$-core.
\end{proof}

We conclude with the main graph-theoretic lemma we need.
\begin{lemma}\label{lem:mainlemma}
Let $G=(V,E)$ be a simple graph that:
\begin{itemize}
\item coincides with its $(3+2)$-core;
\item spans a rigid component $G'=(V',E')$ on  $n'\ge 4$ vertices;
\item and the set of $n''$ vertices $V'' = V\setminus V'$ is incident on at least $2n''$ edges.
\end{itemize}
Then at least one of the following is true
\begin{itemize}
\item $G$ is Laman-spanning
\item $G$ spans a rigid component other than $G'$ on at least $4$ vertices
\end{itemize}
\end{lemma}
\begin{proof}
Pick a Laman basis for $G'$ and discard the rest of the edges spanned by $V'$.  Call the remaining
graph $H$.  Observe that $G$ and $H$ have the same rigid components.
By hypothesis, $H$ now has at least $2n'-3+2n''=2n-3$ edges.  If $H$ is Laman-spanning
we are done, so we suppose the contrary and show that this assumption implies the second conclusion.

Because $H$ is not Laman-spanning and has $2n-3$ edges, it must not be Laman-sparse.  By \lemref{circuits},
$H$ spans a rigid block that is not Laman-sparse, and this block must be contained in some
rigid component $H'$ of $H$.  Finally, since $V'$ induces a Laman-sparse rigid component of $H$ and $H'$ is
a rigid component that isn't Laman-sparse, $H'$ and $G'$ are different rigid components of $H$ and thus $G$.
\end{proof}

\section{Proof of the Main \theoref{main}} \label{sec:thresh}

In this section, we prove our main theorem:
\maintheorem

\paragraph{Roadmap.} The structure of this section is as follows.  We start by establishing that the
constant $c_2$ is the sharp threshold for giant rigid components to emerge.  This is done in two steps:
\begin{itemize}
\item That there is a giant rigid component, w.h.p., when $c>c_2$ is the easier direction, coming from
counting the number of edges in the $3$-core using \propref{FR}.  (\lemref{thresh-gec2})
\item The more difficult direction is that when $c < c_2$ all components are w.h.p. size two or
three is proved using the following idea: if there is a giant rigid component, adding $\Theta(1)$
more random edges will block $2$-orientability, contradicting \propref{FR}. (\lemref{thresh-lec2})
\end{itemize}

The idea of the proof of the size of the giant rigid component is to apply the main
combinatorial \lemref{mainlemma} to the $(3+2)$-core of $G(n,c/n)$ after adding a small
number of uniform edges.  This is possible as a consequence of the more technical
\theoref{almost}. Since only the first conclusion of \lemref{mainlemma} is compatible with \propref{LT}, w.h.p.,
the presence of a large enough giant rigid component follows.
Before we can do that we establish two important structural properties:
\begin{itemize}
\item There is a unique giant rigid component, w.h.p., (\lemref{giant-component-unique})
\item It is contained in the $(3+2)$-core (\lemref{in-3p2-core})
\end{itemize}
With these results, \lemref{giant-component-size} formalizes the plan described above, and \theoref{main} follows.

\paragraph{Sharp Threshold.} We first establish that $c_2$ is the sharp threshold for emergence of giant rigid components.
This is done in the next two lemmas, starting with the more difficult direction.
\begin{lemma}\label{lem:thresh-lec2}
Let $c < c_2$.  Then w.h.p, $G(n,c/n)$ has only rigid components of size at most three.
\end{lemma}
\begin{proof}
\propref{LT} implies that all rigid components in $G(n,c/n)$ have size at most three or are giant.
We will show that, w.h.p., there are no giant rigid components.
Let $\Gamma$ be the event that $G$ spans a rigid component $G'=(V',E')$ on $n'>3$
vertices and $m'$ edges.

Define the graph $H$ to be the one obtained by adding edges sampled
with probability $1/n^2$, independently, from the complement of $G$ in $K_n$.  Since
$H$ is a random graph with edge probability $(c+1/n)/n$, by \propref{FR} $H$ is, w.h.p.,
$2$-orientable so:
\[
\begin{split}
\Pr\left[\text{$H$ is not $2$-orientable}|\Gamma\right]\Pr\left[\Gamma\right] +  \\
\Pr\left[\text{$H$ is not $2$-orientable}|\bar\Gamma\right]\Pr\left[\bar\Gamma\right] = o(1)
\end{split}
\]
We will show that $\Pr\left[\text{$H$ is not $(2,0)$-sparse}|\Gamma\right]$ is uniformly bounded away from
zero, which then forces the probability of a rigid component on more than three vertices to be $o(1)$.

If $\Gamma$ holds, \propref{LT} implies that, $n'=\Omega(n)$, w.h.p.
It follows that, conditioned on $\Gamma$, each of the added edges is
in the span of $V'$ with probability $(1-o(1))\frac{1}{n^2}\Omega(n^2)=\Theta(1)$, so the
probability that at least four of them end up in the span of $V'$ is $\Theta(1)$ as well.
This shows that with probability $\Theta(1)$, the combinatorial lemma
\lemref{adding-edges-to-laman-blocks} applies and so
\[
\Pr\left[\text{$H$ is not $2$-orientable}|\Gamma\right] = \Theta(1).
\]
\end{proof}

\begin{lemma}\label{lem:thresh-gec2}
Let $c > c_2$.  Then w.h.p., $G(n,c/n)$ has at least one giant rigid component.
\end{lemma}
\begin{proof}
By \lemref{circuits}, any simple graph with $n$ vertices and at least $2n-2$ edges spans a
rigid block on at least $4$ vertices.  \propref{FR} implies that for $c > c_2$,
the $3$-core of $G(n,c/n)$ induces such a graph w.h.p.  Finally \propref{LT} implies that
there is a giant rigid component w.h.p.
\end{proof}

\paragraph{Uniqueness of the Giant Rigid Component.} Before we determine the
size, we show that w.h.p. there is only one giant rigid component and that it
is contained in the $(3+2)$-core.

\begin{lemma}\label{lem:giant-component-unique}
Let $c > c_2$.  Then w.h.p., there is a unique giant rigid component in $G(n,c/n)$.
\end{lemma}
\begin{proof}
By \lemref{thresh-gec2}, when $c > c_2$ w.h.p. there is at least one giant rigid component in $G(n,c/n)$.
To show the giant rigid component is unique, we consider $G(n,c/n)$ as being generated by the following
random graph process:
first select a linear ordering of the $\binom{n}{2}$ edges of $K_n$
uniformly at random and take the first $m$ edges from this ordering,
where $m$ has binomial distribution with parameters $\binom{n}{2}$ and $c/n$.

Consider the sequence of graphs $G_1, G_2,\ldots, G_{\binom{n}{2}}$
defined by adding the edges one at a time according to the selected ordering.
Define $t\in [1,\binom{n}{2}]$ to be \emph{critical} if
$G_t$ has one more rigid component spanning more than three vertices than $G_{t-1}$ and \emph{bad} if
$G_t$ has more than one such rigid component.

By \propref{LT}, w.h.p., all rigid components on more than three vertices that appear
during the process have size $\Omega(n)$.  Thus, w.h.p., at most $O(1)$ $t$ are critical.

To bound the number of bad $t$ we note that if $G'=(V',E')$ and $G''=(V'',E'')$ are distinct giant rigid components,
then the probability that a random edge has one endpoint in $V_1\setminus V_2$ and the other
in $V_2\setminus V_1$ is $\Theta(1)$, and the probability that two of the added edges are incident on the same vertex is $O(1/n)$.
So after the addition of $O(\log n)$ random
edges, at least three of them have this property, w.h.p.  \lemref{blocks-combine} then implies that, w.h.p.,
$G'$ and $G''$ persist as giant rigid components for at most $O(\log n)$ steps in the process.   Since there
are at most $O(1)$ such pairs, the total number of bad or critical $t$ is $O(\log n)$, w.h.p.

The probability that $m=t$ is $O(1/\sqrt{n})$,
by standard properties of the binomial distribution. A union bound shows that the probability $m$ is bad
is $O(\log n/\sqrt{n})$, so the probability there is only one rigid component on more than three
vertices and that it is giant is $1-o(1)$ as desired.
\end{proof}

We need two more structural lemmas about the relationship between the giant rigid component and the $(3+2)$-core.
\begin{lemma}\label{lem:in-3p2-core}
Let $c > c_2$.  Then, the unique giant rigid component that exists w.h.p. by \lemref{giant-component-unique}
is contained in the $(3+2)$-core of $G(n,c/n)$.  Moreover, the giant rigid component contains a
(smaller) unique giant rigid block that lies entirely in the $3$-core.
\end{lemma}
\begin{proof}
By \lemref{giant-component-unique}, w.h.p., $G(n,c/n)$ has exactly one rigid component of size at least $4$, and
it spans at least twice as many edges as vertices.  Thus \lemref{component-3plus2core} applies to the rigid
component, w.h.p., so it is in the $(3+2)$-core.  Because the $3$-core itself has average degree at least $4$,
the second part of the lemma follows from the same argument.
\end{proof}

\begin{lemma}\label{lem:3p2-core}
Let $c>c_2$ and suppose that the giant rigid block in the $3$-core implied by \lemref{in-3p2-core}
spans all but $o(n)$ of the vertices in the $3$-core.  Then, w.h.p., the giant rigid component
spans all but $o(n)$ vertices in the $(3+2)$-core.
\end{lemma}
\begin{proof}
Since $c>c_2$, both the $3$-core and the $(3+2)$-core span $\Omega(n)$ vertices, w.h.p.  If the $(3+2)$-core has $o(n)$
more vertices than the $3$-core, then we are already done.  Thus for the rest of the proof, we assume that the
$(3+2)$-core spans $\Omega(n)$ more vertices than the $3$-core.

Let $G$ denote $G(n,c/n)$, let $G_0=(V_0,E_0)$ denote the $3$-core, let $G_1=(V_1,E_1)$ be the $(3+2)$-core,
and let $G'=(V',E')$ be the giant rigid block in the $3$-core.  We now observe that each $v\in V_1-V_0$
sits at the ``top'' of a binary tree with its ``leaves'' in $V_0$.  A branching process argument shows that each
of these has height at most $\log\log n$, w.h.p.  On the other hand, a vertex $v\in V_0$ is in the giant
rigid component when all of these $O(\log n)$ ``leaves'' lie in $V'$.  Since $V_0-V'$ has $o(n)$ vertices,
this happens with probability $1-o(1)$.
\end{proof}

\paragraph{Size of the giant rigid component.}  We are now ready to bound the size of the giant rigid component
when it emerges.  Here is the main lemma.
\begin{lemma}\label{lem:giant-component-size}
Let $c>c_2$.  Then, w.h.p., the unique giant rigid component implied by \lemref{giant-component-unique}
spans a $(1-o(1))$-fraction of the vertices in the $(3+2)$-core.
\end{lemma}
\begin{proof}
Let $G=G(n,c/n)$, let $G_{0}=(V_0,E_0)$ be the $3$-core of $G$, and let $G'=(V',E')$
be a giant rigid block of $G_0$ implied, w.h.p. by \lemref{in-3p2-core}.
Let $n_0$ be the size of $V_0$ and $n'$ the size of $V'$. With high probability, $n_0$ and $n'$
are $\Omega(n)$.

By \theoref{almost} (and this is the hard technical step),
$V''=V_0\setminus V'$ is incident on $2(n_0-n')-O(\log^3 n\, \sqrt{n})$ edges in $G_0$.

Define $H$ to be the graph obtained by adding each edge of $K_n \setminus G$
to $G$ with probability such that $H$ and $G(n,(c+1/n^{5/4})/n)$ are
asymptotically equivalent.

Let $\gamma>0$ be a fixed constant and define $\Gamma$ to be the event that $n_0-n'\ge \gamma n_0$;
i.e., the giant rigid component spans at most a $(1-\gamma)$-fraction of the $(3+2)$-core in $G$.  Conditioned
on $\Gamma$, the expected number of edges added between $V'$ and $V''$ is $\Theta(n^{3/4})$,
so w.h.p., \lemref{mainlemma} applies to $V_0$ in $H$.  Recall that \lemref{mainlemma} has two conclusions:
\begin{itemize}
\item $V_0$ induces a Laman-spanning subgraph
\item $V_0$ spans multiple components spanning at least four vertices in $H$
\end{itemize}
The second case happens with probability $o(1)$ by \propref{LT}, so w.h.p., $V_0$ is Laman-spanning in $H$.

To complete the proof, we note that by \cite{PSW} adding $o(n)$ edges causes the $3$-core to grow by $o(n)$
vertices, w.h.p.  Thus the $3$-core of $G$ spans all but $o(n)$ vertices in the $3$-core of $H$.  Now \lemref{3p2-core}
applies to $H$, showing that, w.h.p., all but $o(n)$ vertices in the $(3+2)$-core lie in the giant rigid component.
\end{proof}

\paragraph{Proof of \theoref{main}.}
The theorem follows from \lemref{thresh-lec2}, \lemref{thresh-gec2}, and \lemref{giant-component-size}.
\hfill $\qed$

\section{Configurations and the algorithmic approach} \label{sec:conf}

We now develop the setting used in the proof of \theoref{almost}, which is based on our analysis of a $2$-orientation
heuristic for random graphs, which is introduced in the next section.  The heuristic operates in the
random configuration model (introduced in~\cite{BC78,B80}), which we briefly introduce here.

\paragraph{Random Configurations.}
A \emph{random configuration} is a model for random graphs with a pre-specified degree sequence; the
given data is a vertex set $V$ and a list of degrees for each vertex such that the sum of the degrees is even.
Define $deg(v)$ to denote the degree of a vertex $v\in V$.

A random configuration is formed from a set $\mathcal{A}$ consisting of $deg(v)$ {\em copies} of each vertex $v \in V$,
defined to be the set of (vertex) copies of $V$. Let $\mathcal{M}$ denote a uniformly random perfect matching in
$\mathcal{A}$. The multigraph $G_\mathcal{A}=(V,E)$ defined by $\mathcal{A}$ and $\mathcal{M}$ has $V$ as its vertex set and
and edge $vw$ for each copy of a vertex $v$ matched to a copy of a vertex $w$.

The two key facts about random configurations that we use here are:
\begin{itemize}
\item Any property that is holds w.h.p. in $G_{\mathcal{A}}$ holds w.h.p. when
conditioned on $G_{\mathcal{A}}$ being 	simple \cite{MR95,MR98}.
\item Any property that holds w.h.p. for $G_{\mathcal{A}}$ in a random configuration with
asymptotically Poisson	degrees holds w.h.p. in the sparse $G(n,c/n)$ model \cite{B01}.
\end{itemize}
Since we are only interested in proving results on $G(n,c/n)$, from now all random configurations discussed have
asymptotically Poisson degree sequences.

\paragraph{The Algorithmic Method.}  Our proof of \theoref{almost} relies on the following observation: a property that
holds w.h.p. for \emph{any} algorithm that generates a uniform matching $\mathcal{M}$ holds w.h.p. for a random configuration.  The
following two moves were defined by Fernholz and Ramachandran \cite{FR07}.
\begin{description}
\item[\textbf{FR I}] Let $\mathcal{A}$ be a set of vertex copies.  Select (arbitrarily) any copy $a_0$ and match it to a copy
$a_1$, selected uniformly at random.  The matching $\mathcal{M}$ is given by the matched pair
$\{a_0,a_1\}$ and a uniform matching on $\mathcal{A}\setminus \{ a_0,a_1\}$.
\item[\textbf{FR II}] Select two copies $a_0$ and $a_1$.  Let $\mathcal{M}'$ be a uniform matching on
$\mathcal{A}\setminus \{ a_0,a_1\}$.  Produce the matching $\mathcal{M}$ as follows:
\begin{itemize}
\item with probability $1/(|\mathcal{A}|-1)$ add the matched pair $\{a_0,a_1 \}$ to $\mathcal{M}'$
\item otherwise, select a matched pair $\{b_0,b_1\}$ uniformly at random and replace it in
$\mathcal{M}'$ with the pairs $\{ a_0,b_0 \}$, $\{ a_1,b_1 \}$.
\end{itemize}
\end{description}
These two moves generate uniform matchings.
\begin{lemma}[{\cite[Lemma 3.1]{FR07}}]\label{lem:matching-moves}
Matchings generated by recursive application of the moves FR I and FR II generate uniform random matchings on
the set of vertex copies $\mathcal{A}$.
\end{lemma}

We will only use the move FRII in the special situation in which $a_0$ and $a_1$ are copies of the same vertex.  With
this specialization, in terms of the graph $G_{\mathcal{A}}$, the two moves correspond to:
\begin{description}
\item[\textbf{FR I}] Reveal an edge of $G_{\mathcal{A}}$ incident on the vertex $v$ that $a_0$ is a copy of.
\item[\textbf{FR II}] Pick two copies $a_0$ and $a_1$ of a vertex $v$.	Generate
$G_{\mathcal{A}-\{a_0, a_1\}}$ and then complete generating $G_{\mathcal{A}}$ by either adding a self-loop to
$v$ or splitting the edge $ij$ corresponding to $\{b_0,b_1\}$ by replacing it with edges $iv$ and $vj$, with probabilities
as described above.
\end{description}

\section{The $2$-orientation algorithm} \label{sec:algo}

We are now ready to describe our $2$-orientation algorithm.  It generates the random multigraph $G_{\mathcal{A}}$
using the moves FR I and FR II in a particular order and orients the edges of $G_{\mathcal{A}}$ as they are
generated.

Since \theoref{almost} is only interested in the $(3+2)$-core, we assume that our algorithm only
runs on the $(3+2)$-core of $G_{\mathcal{A}}$.  Since we can always orient degree two vertices
so that both incident edges point out, the only difficult part of the analysis is the behavior
of our algorithm on the $3$-core of $G_{\mathcal{A}}$.  We denote the set of copies corresponding to the
$3$-core by $\mathcal{A}_s$ and the corresponding multigraph by $G_{\mathcal{A}_s}$.  Define
the number of copies of a vertex $v$ in $\mathcal{A}_s$ by $deg_{\mathcal{A}_s}(v)$.

\paragraph{The $2$-orientation Algorithm.}  We now define our orientation algorithm in terms of the
FR moves.

\alg{$2$-orienting the $3$-core}{
Until $\mathcal{A}_s$ is empty, select a minimum degree vertex $v$ in $\mathcal{A}_s$, and let
$d=deg_{\mathcal{A}_s}(v)$.
\begin{enumerate}
\item If $d\le 2$, execute the FR I move, selecting a copy of $v$ deterministically.  Orient
the resulting edge or self-loop away from $v$.  If $v$ still has any copies left,
call this algorithm recursively, choosing $v$ as the minimum degree vertex.
\item If $d=3$, execute the FR II move, selecting two copies of $v$, and then recursively call this algorithm,
starting the recursive call on $v$ using case $1$.  If the FR II move generated a self-loop, orient it arbitrarily.
If the FR II move split an oriented edge $i\to j$, orient the new edges $i\to v$ and $v\to j$.
\item If $d\ge 4$, perform FR I move on $v$, leaving the resulting edge unoriented.  Then recursively call this algorithm,
choosing $v$ as a minimum degree vertex.
\end{enumerate}
}

We define a vertex $v$ to be \emph{processed} if it is selected deterministically at any time.  A vertex is
defined to be \emph{tight} if the algorithm orients exactly two edges out of it and \emph{loose} otherwise.
For convenience, when a vertex runs out of copies, we simply remove it from $\mathcal{A}_s$; thus
when we speak of the number of remaining vertices, we mean the number of vertices with any
copies left on them.

\paragraph{Correctness and Structural Properties.} We now check that the $2$-orientation algorithm is well-defined.

\begin{lemma}\label{lem:algorithm-correctness}
The $2$-orientation algorithm generates a uniform matching.
\end{lemma}
\begin{proof}
This follows from the fact that it is based on the FR moves and \lemref{matching-moves}.  The only other thing
to check is that if $v$ is being processed and the algorithm is called on $v$ again that $v$ is still a minimum degree
vertex in $\mathcal{A}_s$.  This is true, since $v$ started as minimum-degree and the FR moves decrease its degree at least
as much as any other vertex.
\end{proof}

The following structural property allows us to focus only on the evolution of the degree sequence.
\begin{lemma}\label{lem:loose-vertices}
A loose vertex is $v$ generated only in one of three ways:
\begin{itemize}
\item[\textbf{L1}] $v$ is never processed, because it runs out of copies before it is selected
\item[\textbf{L2}] $v$ has degree one in $\mathcal{A}_s$ when it is processed
\item[\textbf{L3}] $v$ has degree two in $\mathcal{A}_s$ when it is processed, and a self-loop is revealed
\end{itemize}
\end{lemma}
\begin{proof}
The proof is a case analysis.
\begin{itemize}
\item In step 1, if a self-loop is not generated, this is equivalent to the Henneberg I move,
so by \lemref{henneberg} $v$ ends up being tight.
\item Step 2 always corresponds to a Henneberg II move or creates a self-loop, and so by \lemref{henneberg}
$v$ ends up being tight either way, and the out-degree of no other vertex changes.
\item Step 3 cannot leave $v$ with fewer than two copies.
\item The other cases are the ones in the statement of the lemma, completing the proof.
\end{itemize}
\end{proof}

Because of the algorithm's recursive nature, we can, at any time, suspend the algorithm and
just generate a uniform matching on the remaining configuration.  In the next section,
we will use this observation to split the analysis into two parts.
\begin{lemma}\label{lem:suspend}
Let $\mathcal{A}_s|_{t}$ denote the remaining configuration at time $t$.
Suppose that, w.h.p., a random configuration $G^*$ on $\mathcal{A}_s|_{t}$ is $(2,0)$-spanning.
Then, w.h.p., there is an orientation of $G_{\mathcal{A}_s}$ in which all the
vertices of $\mathcal{A}_s|_{t}$ are tight, and the out-degrees of vertices in
$\mathcal{A}_s\setminus \mathcal{A}_s|_t$ are the same as in the full algorithm.
\end{lemma}
\begin{proof}
Since $G^*$ comes from a uniform matching on the remaining copies, and the FR II move
only requires this as input, the cut-off version of the $2$-orientation algorithm
generates a uniform matching.  Thus, w.h.p., results for it hold for $G_{\mathcal{A}_s}$.

By hypothesis, w.h.p., we can orient $G^*$ such that each vertex has out-degree at least two: $G^*$
has a subgraph that is $(2,0)$-spanning and orient the remaining edges arbitrarily, so all
the vertices of $G^*$ are tight, w.h.p.  Moreover, since before $G^*$ is generated, the cut-off
algorithm acts the same way as the full algorithm, which implies that how we orient the edges of $G^*$
does not change the out-degrees in $\mathcal{A}_s\setminus \mathcal{A}_s|_t$.

So the final thing to check is that the FR I and FR II moves don't change the out-degrees in
$G^*$ after the recursive calls return to them.
\begin{itemize}
\item For FR I, no edges induced by $G^*$ are involved, so the statement is trivial.
\item For FR II, either a self-loop is generated, in which case the proof is the same as for FR I; or
an edge is split, in which case the construction used to prove \lemref{henneberg} shows that the
out-degree of vertices in $G^*$ remains unchanged.
\end{itemize}
This completes the proof.
\end{proof}

\paragraph{The Simplified Algorithm.} In light of \lemref{loose-vertices}, we can simplify the analysis
by simply tracking how the degree sequence evolves as we remove copies from $\mathcal{A}_s$.  Since
loose vertices are generated only when a vertex runs out of copies before the algorithm had a chance
to orient two edges out of them, we can ignore the edges and just remove copies as follows.

\alg{Simplified $2$-orientation algorithm}{
Repeatedly execute the following on a vertex $v$ of minimum degree in $\mathcal{A}_s$, until $\mathcal{A}_s$ is empty
\begin{enumerate}
\item \label{step:2} If $deg_{\mathcal{A}_s}(v) \leq 2$, repeat the loop until all copies of $v$ are removed\\
\hspace*{5mm} a. remove a copy of $v$ from $\mathcal{A}_s$. \\
\hspace*{5mm} b. remove a copy chosen uniformly at random from $\mathcal{A}_s$.
\item \label{step:3} If $deg_{\mathcal{A}_s}(v) =3$, first remove the $2$ copies of $v$ from $\mathcal{A}_s$; then execute step~\ref{step:2}.
\item \label{step:i} If $deg_{\mathcal{A}_s}(v) \geq 4$, repeat the following loop $deg_{\mathcal{A}_s}(v)-3$ times and then execute step~\ref{step:3} \\
\hspace*{5mm} a. remove a copy of $v$ from $\mathcal{A}_s$. \\
\hspace*{5mm} b. remove a copy chosen uniformly at random \hspace*{9.5mm} from $\mathcal{A}_s$.
\end{enumerate}
}

As a corollary to \lemref{loose-vertices} we have,
\begin{lemma}\label{lem:simplified-algorithm}
The number of loose vertices generated by the simplified algorithm is exactly the number
of vertices that run out of copies as in L$1$--L$3$ in \lemref{loose-vertices}
\end{lemma}
\begin{proof}
The number of remaining copies of each vertex in $\mathcal{A}_s$ evolves as in the full algorithm.
\end{proof}

\section{Proof of \theoref{almost}} \label{sec:almost}

In this section we prove, \almosttheorem

\paragraph{Roadmap.}  The proof is in two stages: before the minimum numbers of copies on
any vertex remaining reaches four and after.  The overall structure of the argument
is as follows:
\begin{itemize}
\item At the start, the minimum number of copies on any vertex in $\mathcal{A}_s$
is three.  We run the simplified algorithm until the minimum degree
in $\mathcal{A}_s$ rises to four, or the number of vertices remaining reaches $\sqrt{n}$.
\item Since we only run the simplified algorithm when there are
are $\Omega(\sqrt{n})$ copies remaining, we can give bounds on the number
of loose vertices generated by analyzing it as a series of
branching processes.
\item Once the minimum degree in $\mathcal{A}_s$ has reached four, we can
use the combinatorial \lemref{mainlemma} along with the counting
argument \propref{LT} to show that, w.h.p., a random configuration on
what remains is $(2,0)$-spanning.  If it never does, we just declare the
last $\sqrt{n}$ vertices to be loose, so either way we get the desired bound.
\end{itemize}
We define the first stage, when we run the simplified algorithm, to be \emph{phase $3$}; the second stage is
\emph{phase $4$}.
The main obstacle is that during phase 3, there may be many degree two vertices, which increases the
probability of generating loose vertices.  Let us briefly sketch our approach.

At the start of phase 3, every vertex has degree at least three (and $\Omega(n)$ have degree three, w.h.p.).
The algorithm will:
\begin{itemize}
\item Pick a degree three vertex, remove all of its copies and a random copy.
\item Removing the random copy may create a degree two vertex, which is then removed, along with
two random copies.
\item These random copies may create more degree two vertices or even a loose vertex.
\end{itemize}
We call the cascade described above a \emph{round} of the algorithm.  We model each round as a
\emph{branching process}, which we analyze with a system of differential equations (this step
occupies most of the section).  The key fact is that, with high probability, all rounds process $O(\log n)$ vertices,
which is enough to bound the number of loose vertices using arguments similar to those from the
high-degree phase analysis.

\paragraph{Phase 3: The Main Lemma.}  We start with the analysis of phase $3$.  In phase $3$,
all iterations of the algorithm take step 1 or 2.  We define a \emph{round} of the algorithm in phase $3$
to start with a step 2 iteration and contain all the step 1 iterations before the next step 2.  The critical lemmas
are.
\begin{lemma}\label{lem:round-length}
With probability at least $1-1/n$, all rounds during phase $3$ process  $O(\log n)$ vertices.
\end{lemma}

We defer the proofs for now, and instead show how they yield a bound on the number of loose
vertices.

\begin{lemma}\label{lem:dense-phase-3}
With high probability, the expected number of loose vertices generated in a phase $3$ round is $O(\log^2 n/\sqrt{n})$.
\end{lemma}
\begin{proof}
All vertices start the round with at least three copies on them, so \lemref{loose-vertices} implies that
a loose vertex is either
\begin{itemize}
\item hit at least twice at random (L1 and L2)
\item reaches degree two and then gets a self-loop (L3)
\end{itemize}

Both of these events happen with probability $O(1/\sqrt{n})$, since we are only running the algorithm
with this many vertices left, and the
probability any specific vertex is hit is $O(1/k)$, where $k$ is the number of remaining
copies, and w.h.p. this is $\Omega(n)$ during any phase 3 round.

Since the round lasts $O(\log n)$ iterations, w.h.p.,
the probability any vertex becomes loose is at most $O(\log n/\sqrt{n})$.  Because any
loose vertex generated in a round must be processed in that round, the expected number of
loose vertices generated is $O(\log n)O(\log n/\sqrt{n})=O(\log^2 n/\sqrt{n})$.
\end{proof}

\begin{lemma}\label{lem:low-degree}
With high probability, the number of loose vertices generated during phase $3$ is $O(\log^3 n\sqrt{n})$.
\end{lemma}
\begin{proof}
Let $\Gamma$ be the event that all rounds process $O(\log n)$ vertices.  The main  \lemref{round-length}
implies that $\Gamma$ fails to hold with probability $1/n$, so we are done if the lemma holds when conditioning
on $\Gamma$.

Assuming, $\Gamma$, \lemref{dense-phase-3} applies to $O(n)$ rounds.  From
\lemref{dense-phase-3}, we see that, w.h.p., the expected number of loose vertices generated in
phase $3$ is $O(n)O(\log^2/\sqrt{n})=O(\log^2 n\sqrt{n})$.  It then follows from Markov's inequality that the probability
of $\Omega(\log^3 n\sqrt{n})$ loose vertices generated in phase $3$ is at most $1/\log n=o(1)$.

If phase $3$
runs until it hits its cutoff point, then there are an additional $O(\sqrt{n})$ loose vertices, but this preserves the desired bound.
\end{proof}

\paragraph{Phase 4.} At the start of phase $4$, we stop the simplified algorithm, since we can
prove the following lemma more directly.
\begin{lemma}\label{lem:high-degree}
With high probability, at the start of phase $4$, a uniform simple graph generated from $\mathcal{A}_s$
is $(2,0)$-spanning.
\end{lemma}
\begin{proof}
Since the degree sequence of $\mathcal{A}_s$ has minimum degree $4$ and is truncated Poisson, a simple
graph $G^*$ generated by $\mathcal{A}_s$ is asymptotically equivalent to the $4$-core of some $G(n,c'/n)$.
By edge counts, \lemref{mainlemma} applies to it, and the conclusion in which it is not Laman-spanning is
ruled out, w.h.p., by \propref{LT}.  Since $G^*$ is a Laman graph plus at least three more edges, the main
theorem of \cite{maps} implies that $G^*$ is $(2,0)$-spanning.
\end{proof}

\paragraph{Proof of \theoref{almost}}
Because we are interested in results on $G(n,c/n)$, we condition on the event that
$G_{\mathcal{A}_s}$ is simple. By \lemref{low-degree}, w.h.p., phase $3$ generates $O(\log^3 n\sqrt{n})$ loose vertices.
By \lemref{high-degree} (which applies when $G_{\mathcal{A}_s}$ is simple), w.h.p., we can apply \lemref{suspend} to
conclude that there is an orientation with no loose vertices that were not generated during phase~$3$.
\hfill $\qed$

\paragraph{Further Remarks.} The assumption that $G_{\mathcal{A}_s}$ is simple can be removed
at the notational cost of introducing $(2,0)$-blocks and components and proving the appropriate generalization
of \lemref{mainlemma} to that setting.  Since the added generality doesn't help us here, we leave this
to the reader.

In the rest of this section we prove \lemref{round-length}.

\paragraph{The Branching Process.}

Consider a round in phase $3$.  Its associated branching process is defined as follows:
\begin{itemize}
\item The root vertex of the process is the degree three vertex processed to start the round.
\item The children of any vertex $v$ processed in the round are any degree one or two vertices created
while processing $v$.
\end{itemize}
\figref{loose-example} shows an example of the process, and how a loose vertex is generated.
\begin{figure}[htbp]
\centering
\includegraphics[width=0.35\textwidth]{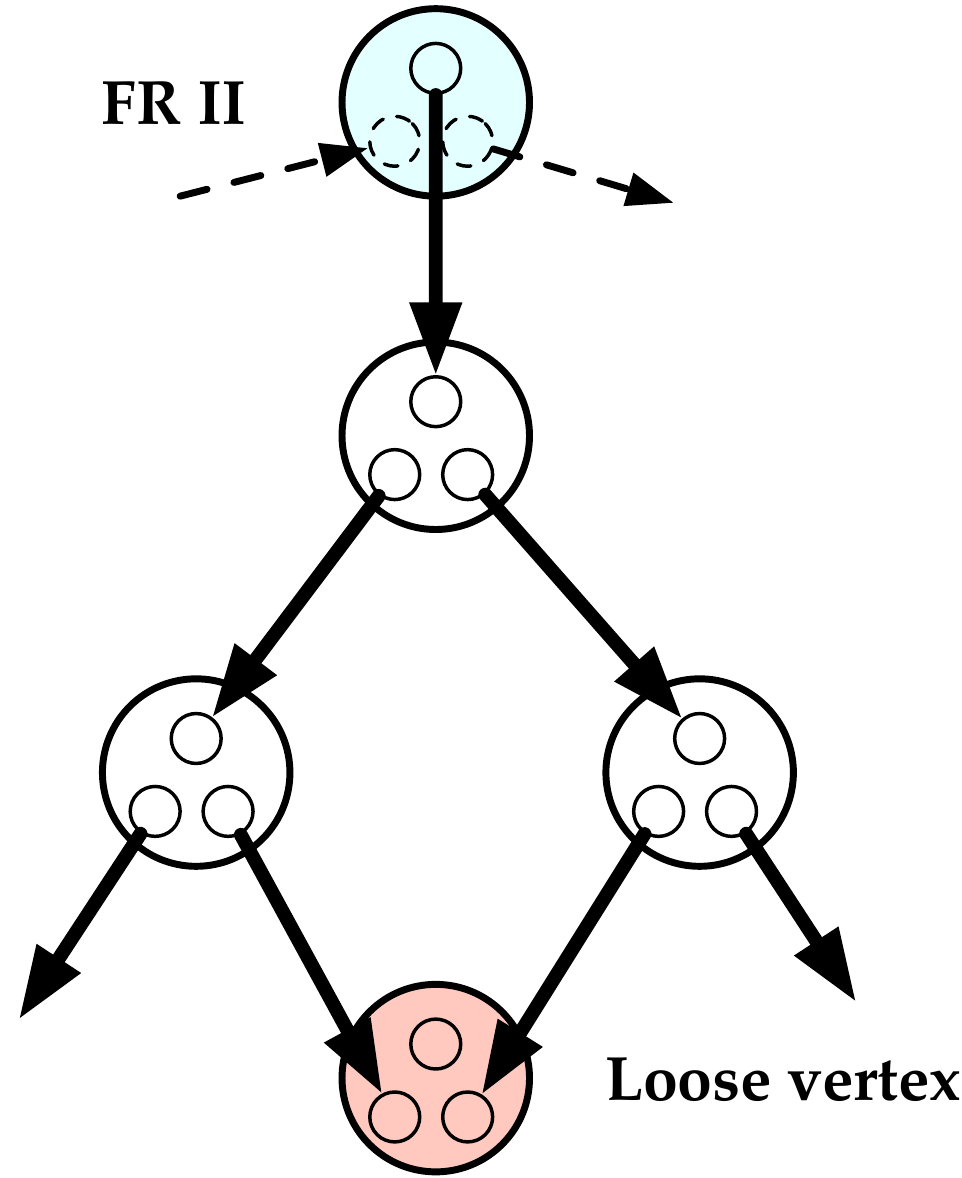}
\caption{Example of the Phase $3$ branching process.  A loose vertex is generated when it is hit twice by a
previous generation.}
\label{fig:loose-example}
\end{figure}
We define $\lambda$ to be the expected number of children.  We will show,
\begin{lemma}\label{lem:subcritical}
With high probability, $\lambda < 1$ in all phase $3$ rounds.
\end{lemma}
This implies the key \lemref{round-length}.
\begin{proof}[Proof of \lemref{round-length}]
If $\lambda < 1$ in all rounds, then standard results on branching processes imply that the probability
any particular round processes more than $O(\log n)$ vertices is $O(1/n^2)$, for appropriate choices
of constants.  A union bound then implies that the probability all of them process $O(\log n)$ vertices is at least
$1-1/n$.

This assumption on $\lambda$ holds w.h.p. by \lemref{subcritical}, completing the proof.
\end{proof}

\paragraph{Analysis of the Branching Process: Proof of \lemref{subcritical}.} We will use the method of
differential equations \cite{W95} to establish \lemref{subcritical}.  All the required Lipchitz conditions
and tail bounds in our process are easy to check, since the degree sequence is asymptotically Poisson.  Thus
the main step is to define and analyze the system of differential equations describing the evolution of the
degree sequence as the algorithm runs.

In order to simplify the analysis, we break the loop into individual {\em timesteps}. At each timestep, two vertex copies are removed from $\mathcal{A}_s$. For example, the step~\ref{step:2} of the algorithm is divided into $deg_{\mathcal{A}_s}(v)$ many timesteps, where in each timestep we remove a copy of $v$ and another random copy. Similarly, step~\ref{step:i} of the algorithm gets divided into $deg_{\mathcal{A}_s}(v)-3$ many timesteps.

For all $i \ge 3$, let $a^{(i)}$ denote the number of vertices of degree $i$ in the $3$-core divided by $n$ (where $n$ is the number of vertices in $G=G(n,c/n)$). We say a vertex is {\em hit} during a round in phase 3 if a copy of the vertex is selected by step~\ref{step:2}(b) of the algorithm during the round. Since in phase $3$ vertices of degree $4$ and greater get hit randomly with probability proportional to their degree, they always obey a truncated Poisson  distribution. Therefore, for $i \ge 4$, $a^{(i)}$ will be a truncated Poisson distribution with a time-varying mean $\delta$,
\begin{equation}
\label{eq:trunc-poi}
a^{(i)} = \frac{\e^{-\delta} \,\delta^i}{i!} \,.
\end{equation}
Let $\mu$ denote the total number of vertex copies in the $3$-core divided by $n$. That is,
\[
\begin{split}
\mu
= \sum_{i \ge 3} i a^{(i)}
= 3 a^{(3)} + \sum_{i \ge 4} \frac{\e^{-\delta} \,\delta^i}{(i-1)!} \\
= 3 a^{(3)} + \delta \left( 1-\e^{-\delta} \left(1+\delta+\frac{\delta^2}{2} \right) \right) \, .
\end{split}
\]
We will use the following notation for the number of copies of degree $j$ or greater divided by $n$, $\mu^{\ge j} = \sum_{i=j}^\infty i a^{(i)} \, . $

Since an edge hits a vertex of degree $3$ with probability $3a^{(3)}/\mu$, and since it creates two new edges when it does so, the branching ratio of the
branching process defined above is
$
\lambda = 6a^{(3)}/\mu.
$
To show that the branching process is subcritical we need to analyze $\lambda$ and show that throughout phase $3$ it is bounded away from $1$.

Now, in one timestep the expected change in $a^{(i)}$ ($i \geq 4$) is
$$n \cdot \E[\Delta a^{(i)}] =  \frac{(i+1)a^{(i+1)}}{\mu} - \frac{ia^{(i)}}{\mu} \,.$$
Here, the first term on the right hand side represents the probability that a degree $i+1$ vertex is hit (this creates a new degree $i$ vertex). The second term represents the probability that a degree $i$ vertex is hit (this destroys a degree $i$ vertex).  Similarly, the expected change in $a^{(3)}$ in one timestep is
$$n \cdot \E[\Delta a^{(3)}] = \frac{4a^{(4)}}{\mu} - \frac{3a^{(3)}}{\mu} - (1 - \lambda) \,.$$
Here, the additional $1-\lambda$ term comes due to the FR II step. The expected probability on a given timestep that there are no degree-$2$ vertices is $1-\lambda$.

Setting the derivatives of these variables to their expectations gives us our differential equations:

\[
n \cdot \frac{\da^{(3)}}{\dt}
= \frac{-3a^{(3)} + 4a^{(4)}}{\mu} - (1-\lambda)
\]
and, for all $i\ge 4$
\[
n \cdot \frac{\da^{(i)}}{\dt}
= \frac{-i a^{(i)} + (i+1) a^{(i+1)} }{\mu}
\]
Changing the variable of integration to $s$ where $\ds/\dt = 1/(n \mu)$ and substituting $\lambda =6a^{(3)}/\mu$ gives
\begin{align}
\frac{\da^{(3)}}{\ds}
&= -3a^{(3)} + 4a^{(4)} - \mu + 6a^{(3)} = -\mu^{\ge 5} \label{eq:da3ds} \\
\frac{\da^{(i)}}{\ds}
&= -i a^{(i)} + (i+1) a^{(i+1)} \quad \mbox{for all $i \ge 4$} \label{eq:daids} \, .
\end{align}
This infinite system is consistent with the truncated Poisson form for $a^{(i)}$.  Since~\eqref{eq:trunc-poi} implies
\[
\frac{\da^{(i)}}{\ds} = \left( \frac{i}{\delta} - 1 \right) a^{(i)} \,\frac{\ddelta}{\ds}
\; \;
\mbox{ and } \; \;
a^{(i+1)} = \frac{\delta}{i+1} \, a^{(i)} \, ,
\]
Equation~\eqref{eq:daids} becomes
\[
\left( \frac{i}{\delta} - 1 \right) a^{(i)} \,\frac{\ddelta}{\ds}
= ( -i + \delta ) a^{(i)}
\; \; \mbox{ or } \;\; \frac{\ddelta}{\ds} = -\delta  \, .\]

Thus, $\delta$ as a function of $s$ is, $\delta = \delta_0 \,\e^{-s}$. Here, $\delta_0$ is the parameter of the truncated Poisson distribution describing the degree distribution of the $3$-core.  We can also express $\mu$ as a  function of $s$. Since $\lambda=6a^{(3)}/\mu$, we can also express $\lambda$ as a function of $s$. Then we show using analytic arguments that  $\lambda < 1$ for all values of $s$ in phase $3$.

We start by computing the integral,
\[
a^{(3)} =
a^{(3)}_0
- \int_0^s \mu^{\ge 5} \,\ds,
\]
analytically. Since $\frac{\ddelta}{\ds} = -\delta$, we have $\ds = \frac{-\ddelta}{\delta}$. Let $\delta_0=\tau$. Remember, $\delta_0$ is the parameter of the truncated Poisson distribution describing the degree distribution of the $3$-core. Now, $\mu^{\geq 5}$ can be expressed as
$$ \mu^{\geq 5} = \sum_{i=5}^{\infty} \frac{\e^{-\delta} \delta^i}{(i-1)!} = \delta \left (1 - \e^{-\delta} \left (1+\delta+ \frac{\delta^2}{2} + \frac{\delta^3}{6} \right ) \right ).$$ We can compute $a^{(3)}$ as
\begin{align*}
\lefteqn{a^{(3)} } \\ & = a^{(3)}_0
- \int_0^s \mu^{\ge 5} \, \ds   =   a^{(3)}_0 + \\ & \int_{\tau}^{\tau \e^{-s}} \delta  \left (1 - \e^{-\delta} \left (1+\delta+ \frac{\delta^2}{2} + \frac{\delta^3}{6} \right ) \right ) \frac{\ddelta}{\delta} \\
& =  a^{(3)}_0 + \int_{\tau}^{\tau \e^{-s}} \left (1 - \e^{-\delta} \left (1+\delta+ \frac{\delta^2}{2} + \frac{\delta^3}{6} \right ) \right ) \ddelta \\
& = a^{(3)}_0 + \tau\e^{-s} + \e^{-\tau\e^{-s}} \left ( \frac{\tau^3\e^{-3s}}{6}+\tau^2\e^{-2s} + 3\tau\e^{-s}+4 \right ) \\ & - \left ( \tau+ \frac{\tau^3\e^{-\tau}}{6}+\tau^2\e^{-\tau}+3\tau\e^{-\tau}+4\e^{-\tau} \right ). \end{align*}
Now, $a^{(3)}_0 = \e^{-\tau} \tau^3 /6$. Therefore, $a^{(3)}$ as a function of $s$ is
\begin{align*} \label{eqn:lhs} \lefteqn{a^{(3)}} \\  &  = \frac{\e^{-\tau}\tau^3}{6}+\tau\e^{-s} + \e^{-\tau\e^{-s}}( \frac{\tau^3\e^{-3s}}{6}+\tau^2\e^{-2s} + 3\tau\e^{-s}+4 )  \\ &  - \left ( \tau+ \frac{\tau^3\e^{-\tau}}{6}+\tau^2\e^{-\tau}+3\tau\e^{-\tau}+4\e^{-\tau} \right ). \end{align*}
Since both $a^{(3)}$ and $\delta$ can be expressed as a function of $s$, we can write
\begin{align*}
\lefteqn{\mu} \\
& =3 a^{(3)} + \delta \left( 1-\e^{-\delta} \left(1+\delta+\frac{\delta^2}{2} \right) \right) \\
& = 3\left ( \frac{\e^{-\tau}\tau^3}{6}+\tau\e^{-s}  \right ) \\
& + 3 \left ( \e^{-\tau\e^{-s}} \left ( \frac{\tau^3\e^{-3s}}{6}+\tau^2\e^{-2s} + 3\tau\e^{-s}+4 \right ) \right ) \\
& - 3 \left ( \tau+ \frac{\tau^3\e^{-\tau}}{6}+\tau^2\e^{-\tau}+3\tau\e^{-\tau}+4\e^{-\tau} \right )\\
& + \tau\e^{-s} \left ( 1 - \e^{-\tau\e^{-s}} \left ( 1 + \tau\e^{-s}+\frac{\tau^2\e^{-2s}}{2} \right ) \right ),
\end{align*}
as a function of $s$. Now, given $a^{(3)}$ and $\mu$ as functions of $s$, we can compute $\lambda=6a^{(3)}/\mu$ as a function of $s$. We need to show that $\lambda < 1$ for all values of $s \geq 0$ in phase $3$. Let $s^*$ be the value of $s$ at which $a^{(3)}$ (as a function of $s$) evaluates to $0$ (this signals the end of phase 3). By taking derivatives, it can be shown that $\lambda$ is a decreasing function of $s$ in the interval $(0,s^*)$ (we omit this calculation here). So, all left to verify is whether at $s=0$, $\lambda < 1$. Evaluating $\lambda$ at $s=0$, we get
$$\lambda|_{s=0} =  \frac{6 \cdot \frac{\e^{-\tau}\tau^3}{6}}{\frac{3\e^{-\tau}\tau^3}{6} + \tau - \e^{-\tau}(\tau+\tau^2+\frac{\tau^3}{2})} = \frac{\e^{-\tau}\tau^2}{1-\e^{-\tau}(1+\tau)}.$$
For $\tau> 1.794$, $\lambda|_{s=0} < 1$. At the birth of the giant rigid component $(c=c_2 \approx 3.58804)$, $\tau = 2.688$\footnote{Remember, $\tau (= \delta_0)$ is the parameter of the truncated Poisson distribution describing the degree distribution of the $3$-core. The value of $\tau$ equals $qc$ where $q=1-\e^{-qc}(1+qc)$ is the fraction of the vertices in the $3+2$-core of $G$ (see, e.g.,~\cite{PSW}). For $c=3.58804$, we get  $q=0.749154$ and $\tau = 2.688$.}, and as we increase $c$, the value of $\tau$ only increases. Therefore, for all $c> c_2$, $\lambda|_{s=0} < 1$. As, $\lambda$ is a decreasing function of $s$, therefore for $c> c_2$, $\lambda < 1$ throughput phase $3$. This shows that for $c > c_2$ all branching processes in phase $3$ are subcritical.

\hfill $\qed$

\section{Conclusions} \label{sec:conclusions}

We studied the emergence of rigid components in sparse Erd\H{o}s-R\'enyi random graphs, proving that there
is a sharp threshold and a quantitative bound on the size of the rigid component when it emerges.  These
results confirm theoretically the simulations of \cite{rivoire2006exactly}.

As conjectures, we leave the following:

\begin{Conjecture}
With high probability, the entire $(3+2)$-core is Laman-spanning in $G(n,c/n)$, for $c>c_2$.
\end{Conjecture}

\begin{Conjecture} With high probability, the $3$-core of $G(n,c/n)$ is \emph{globally rigid} (i.e., there is \emph{exactly one} embedding of the framework's edge lengths, as opposed to a discrete set), for $c>c_2$.
\end{Conjecture}
In the plane, generic global rigidity is characterized by the framework's graph being $3$-connected and remaining Laman-spanning
if any edge is removed \cite{C05}. Proving either of these conjectures using the plan presented here
would most likely require a stronger statement than \theoref{almost}, such
as the $3$-core being, w.h.p., $(2,0)$-spanning.
\begin{Conjecture}
With high probability, the size of the $(3+2)$-core when it emerges is given by $qn$, where $q$
is the root of the equation $q = 1-\e^{-qc} (1+qc)$.
\end{Conjecture}
The term $\e^{-qc} (1+qc)$ comes from $\Pr[Po(qc) < 2]$. This conjecture, with \theoref{main}, implies that when the giant rigid component emerges it spans about $0.749n$ vertices.

\paragraph{Acknowledgements.}  We would like to thank Michael Molloy, Alexander Russell, and Lenka Zdeborov\`a for initial discussions.
We would also like to thank Daniel Fernholz and  Vijaya Ramachandran for discussions regarding~\cite{FR07}.

\end{document}